\documentclass{amsart}

\usepackage[latin1]{inputenc}
\usepackage{amsfonts}
\usepackage{eufrak}
\usepackage{amsmath}
\usepackage{amsthm}
\usepackage{a4wide}
\usepackage[all]{xy}
\usepackage{textcomp}

\newtheorem{tthm}{Theorem}[section]
\theoremstyle{defi}

\theoremstyle{ddefi}

\theoremstyle{prop}
\newtheorem{prop}[tthm]{Proposition}
\theoremstyle{pprop}
\newtheorem{pprop}[tthm]{Proposition}
\theoremstyle{rmk}

\theoremstyle{rrmk}

\theoremstyle{lemma}

\theoremstyle{llemma}
\newtheorem{llemma}[tthm]{Lemma}

\theoremstyle{ex}

\theoremstyle{cor}

\theoremstyle{qq}

\begin{document}

\title[Equivariant Homology and Orderability of Lens Spaces]
{Equivariant Homology for Generating Functions and Orderability of Lens Spaces}

\author{Sheila Sandon}
\address{Laboratoire de Math\'{e}matiques Jean Leray, Universit\'{e} de Nantes, 44322 Nantes, France}
\email{sheila.sandon@univ-nantes.fr}

\begin{abstract}
In her PhD thesis \cite{Milin} Milin developed a $\mathbb{Z}_k$-equivariant version of the contact homology groups constructed in \cite{EKP} and used it to prove a $\mathbb{Z}_k$-equivariant contact non-squeezing theorem. In this article we re-obtain the same result in the setting of generating functions, starting from the homology groups studied in \cite{mio}. As Milin showed, this result implies orderability of lens spaces.
\end{abstract}

\maketitle

\section{Introduction}

Let $(P,\eta)$ be a (cooriented) contact manifold. A contact isotopy $\phi_t$ of $P$ is said to be positive (respectively non-negative) if it is generated by a positive (non-negative) contact Hamiltonian. Geometrically this means that $\phi_t$ moves every point of $P$ in a direction which is positively transverse (or tangent) to the contact distribution. Consider now the universal cover $\widetilde{\text{Cont}_0(P,\xi)}$ of the identity component of the group of contactomorphisms of $P$. Recall that $\widetilde{\text{Cont}_0(P,\eta)}$ is defined to be the set of homotopy classes of paths $\phi_t$ of contactomorphisms of $P$ with fixed endpoints $\phi_0=\text{id}$ and $\phi_1$. The group structure can be defined equivalently either by composition of contactomorphisms or by juxtaposition of paths. We define a relation $\leq$ by saying that $[\phi_t]\leq [\psi_t]$ if the class $[\psi_t][\phi_t]^{-1}$ can be represented by a non-negative contact isotopy. This relation is always reflexive and transitive. The contact manifold $(P,\eta)$ is said to be \textit{orderable} if anti-symmetry also holds. In this case $\leq$ defines a (bi-invariant) partial order on $\widetilde{\text{Cont}_0(P,\eta)}$.\\
\\
The notion of orderability was introduced by Eliashberg and Polterovich in \cite{EP}. In the same paper the authors also proved that a contact manifold is orderable if and only if there are no positive contractible loops in the identity component of the group of contactomorphisms, and used this criterion to infer orderability of projective space  $\mathbb{R}P^{2n-1}$ from Givental's theory of the non-linear Maslov index \cite{G}. Since then other contact manifolds have been proved to be orderable: standard contact Euclidean space \cite{B}, 1-jet bundles \cite{CFP,CN}, cosphere bundles \cite{EKP,CN09} and lens spaces \cite{Milin}. On the other hand, the first examples of non-orderable contact manifolds were found by Eliashberg, Kim and Polterovich \cite{EKP}. In particular they showed that the orderability question is sensitive to the topology of the underlying manifold, by proving that the standard contact sphere $S^{2n-1}$ (for $n>1$) is non-orderable even though it is the double cover of the orderable contact manifold $\mathbb{R}P^{2n-1}$. Another surprising result contained in \cite{EKP} is the discovery that orderability is related to some contact squeezing and non-squeezing phenomena. \\
\\
Given open domains $\mathcal{U}_1$ and $\mathcal{U}_2$ in a contact manifold $(V,\xi)$ we say that $\mathcal{U}_1$ can be squeezed into $\mathcal{U}_2$ if there exists a contact isotopy $\phi_t$, $t\in [0,1]$, from the closure $\overline{\mathcal{U}_1}$ of $\mathcal{U}_1$ to $V$ such that $\phi_0$ is the identity and $\phi_1\,(\overline{\mathcal{U}_1})\subset \mathcal{U}_2$. The key to relate the two notions of orderability and contact squeezing is a geometric construction \cite{EKP} that, given a Liouville manifold $(M,\omega=-d\lambda)$ with non-orderable ideal contact boundary $\big(\,P,\eta=\text{ker}(\lambda)\,\big)$, uses a positive contractible loop of contactomorphisms of $P$ to squeeze domains in the prequantization $\big(V=M \times S^1, \xi=\text{ker}(d\theta-\lambda)\big)$ of $M$. A particular application of this idea is the following squeezing result in the contact manifold $\big(\mathbb{R}^{2n}\times S^1,\text{ker}\,(dz-\frac{ydx-xdy}{2})\big)$. We will denote by $B(R)$ the domain $\{\,\pi \sum_{i=1}^n x_i^{\,2} + y_i^{\,2} <R\,\}$ in $ \mathbb{R}^{2n}$. Moreover, given a domain $\mathcal{U}$ of $\big(\mathbb{R}^{2n},\omega_0=dx\wedge dy\big)$ we will denote by $\widehat{\mathcal{U}}$ its prequantization $\mathcal{U} \times S^1$ in $\mathbb{R}^{2n} \times S^1$.

\begin{tthm}[\cite{EKP}]\label{squeezing}
If $n>1$ then any domain of the form $\widehat{B(R)}$ with $R<1$ can be squeezed into $\widehat{B(R')}$ for $R'$ arbitrarily small.
\end{tthm}

This result is proved by applying the squeezing construction to the special case of $M=\mathbb{R}^{2n}$ and using an explicitly constructed positive contractible loop of contactomorphisms of $S^{2n-1}$. Note that the proof does not work for $n=1$ because $S^1$ does not admit any positive contractible loop of contactomorphisms \footnote{Using the techniques in \cite{E} it can be proved that in dimension $2n+1=3$ it is never possible to squeeze $\widehat{B(R)}$ into a smaller $\widehat{B(R')}$. Some contact rigidity is also present in higher dimension: it was discovered by Eliashberg, Kim and Polterovich \cite{EKP} (and reproved in \cite{mio} using generating functions techniques) that $\widehat{B(R)}$ cannot be squeezed into $\widehat{B(R')}$ if $R'\leq k \leq R$ for some integer $k$.}. On the other hand, the same squeezing construction can also be used in the opposite direction, i.e. to infer non-existence of a positive contractible loop of contactomorphisms of $P$ (hence orderability of $P$) by proving non-squeezing results in $M \times S^1$. In particular, this was the strategy followed by Milin \cite{Milin} to prove orderability of lens spaces.\\
\\
Consider the lens space $L_k^{\phantom{k}2n-1}(m_0,\cdots,m_{n-1})$, which is defined to be the quotient of the unit sphere $S^{2n-1}$ in $\mathbb{C}^n$ by the $\mathbb{Z}_k$-action generated by the map 
\begin{equation}\label{e_action}
\tau_k: (w_0, \cdots, w_{n-1}) \mapsto \big(e^{2\pi im_0/k}w_0, \cdots, e^{2\pi im_{n-1}/k}w_{n-1} \big).
\end{equation}
In the following the coefficients $m_0$,$\cdots$, $m_{n-1}$ will play no role so we will drop them from the notation. Since the standard contact form $\lambda_0=\frac{ydx-xdy}{2}$ on $S^{2n-1}$ is invariant with respect to this action it descends to a contact form on $L_k^{\phantom{k}2n-1}$. To prove orderability of $L_k^{\phantom{k}2n-1}$, by the criterion given in \cite{EP} we need to show that there is no positive contractible loop in the identity component $\text{Cont}_0\big(L_k^{\phantom{k}2n-1}\big)$ of its contactomorphisms group. Note that the identity component $\text{Cont}_0^{\phantom{0}\mathbb{Z}_k}\big(S^{2n-1}\big)$ of the group of $\mathbb{Z}_k$-equivariant contactomorphisms of $S^{2n-1}$ is a connected $k$-fold covering of $\text{Cont}_0\big(L_k^{\phantom{k}2n-1}\big)$, and so it is enough to show that there is no positive contractible loop in $\text{Cont}_0^{\phantom{0}\mathbb{Z}_k}\big(S^{2n-1}\big)$. Note also that without loss of generality we may assume that $k$ is prime. \\
\\
Elaborating on the construction of Eliashberg, Kim and Polterovich, Milin showed that if there was a positive contractible loop in $\text{Cont}_0^{\phantom{0}\mathbb{Z}_k}\big(S^{2n-1},\xi\big)$ then this could be used as a tool for $\mathbb{Z}_k$-equivariant contact squeezing of domains in $\mathbb{R}^{2n}\times S^1$, where $\mathbb{Z}_k$ acts on $\mathbb{R}^{2n}\times S^1$ by the action (\ref{e_action}) in each $\mathbb{R}^{2n}$-fiber. More precisely, such a positive contractible loop could be used to squeeze $\widehat{B(R)}$ for some small $R$ into $\widehat{B(R')}$ for $R'$ arbitrarily small, by a compactly supported $\mathbb{Z}_k$-equivariant contact isotopy of $\mathbb{R}^{2n}\times S^1$. Thus, orderability of $L_k^{\phantom{k}2n-1}$ follows if we prove the following equivariant contact non-squeezing theorem.

\begin{tthm}\label{non-squeezing}
Given any $R$, the domain $\widehat{B(R)}$ cannot be squeezed by a $\mathbb{Z}_k$-equivariant contact isotopy of $\mathbb{R}^{2n}\times S^1$ into $\widehat{B(R')}$ for $R'$ arbitrarily small.
\end{tthm}

Theorem \ref{non-squeezing} was proved by Milin \cite{Milin} using a $\mathbb{Z}_k$-equivariant version of the theory of contact homology for domains of $\mathbb{R}^{2n}\times S^1$ that was developed in \cite{EKP}. In this article we will reprove Theorem \ref{non-squeezing} in the setting of generating functions, starting from the contact homology groups studied in \cite{mio}.\\
\\
In \cite{mio} we used the set-up of Bhupal \cite{B} to generalize to the contact case Traynor's construction \cite{T} of symplectic homology for Hamiltonian symplectomorphisms and domains in $\mathbb{R}^{2n}$. In this way we obtained homology groups $G_{\ast}^{\;\;(a,b]}\,(\phi)$ for compactly supported contactomorphisms of $\mathbb{R}^{2n}\times S^1$ isotopic to the identity, and homology groups $G_{\ast}^{\;\;(a,b]}\,(\mathcal{V})$ for domains of $\mathbb{R}^{2n}\times S^1$. Here $a$ and $b$ are \textit{integer} parameters. The group $G_{\ast}^{\;\;(a,b]}\,(\phi)$ is defined to be the relative homology of the sublevel sets $E^a$ and $E^b$ of a generating function $S:E\rightarrow\mathbb{R}$ of $\phi$, while $G_{\ast}^{\;\;(a,b]}\,(\mathcal{V})$ is obtained by considering the inverse limit of the homology groups of contactomorphisms supported in $\mathcal{V}$. In this article we will prove that if $\phi$ is a $\mathbb{Z}_k$-equivariant contactomorphism then it has a generating function $S:(\mathbb{R}^{2n}\times S^1)\times\mathbb{R}^{M(2n+1)}\rightarrow\mathbb{R}$ which is invariant with respect to the diagonal $\mathbb{Z}_k$-action on $(\mathbb{R}^{2n}\times S^1)\times\mathbb{R}^{M(2n+1)}$ (Proposition \ref{exist}). Moreover we will show that all such generating functions for $\phi$ are related by $\mathbb{Z}_k$-equivariant fiber preserving diffeomorphisms and $\mathbb{Z}_k$-invariant stabilization (Proposition \ref{equiv_uniq}). We will then define the $\mathbb{Z}_k$-equivariant contact homology group $G_{\mathbb{Z}_k,\,\ast}^{\;\;\;(a,b]}\,(\phi)$ of $\phi$ to be the equivariant relative homology of the sublevel sets of $S$. The $\mathbb{Z}_k$-equivariant contact homology $G_{\mathbb{Z}_k,\,\ast}^{\;\;\;(a,b]}\,(\mathcal{V})$ of a $\mathbb{Z}_k$-invariant domain $\mathcal{V}$ in $\mathbb{R}^{2n}\times S^1$ is then defined via a limit process, as in the non-equivariant case.\\
\\
The existence and uniqueness results for generating functions of $\mathbb{Z}_k$-equivariant contactomorphisms of $\mathbb{R}^{2n}\times S^1$ are obtained by adapting the arguments given in the non-equivariant case by Chaperon \cite{Ch} and Th\'eret \cite{Th, Th2}. An important ingredient is the new formula (\ref{formula}) for the Legendrian submanifold $\Gamma_{\phi}$ of $J^1(\mathbb{R}^{2n+1})$ associated to a contactomorphism $\phi$ of $\mathbb{R}^{2n+1}$. Indeed, the formula that was introduced by Bhupal \cite{B}, and used in \cite{mio}, does not preserve the $\mathbb{Z}_k$-action.\\
\\
In order to prove Theorem \ref{non-squeezing} we will need to calculate the equivariant homology of balls (Theorem \ref{thm_balls_eq}). We will do this by elaborating on the calculations given by Traynor. In \cite{T} she constructed an unbounded ordered sequence of Hamiltonian symplectomorphisms supported in a ball $B(R)$ in $\mathbb{R}^{2n}$, and obtained the symplectic homology of $B(R)$ by calculating the index and critical value of the critical submanifolds of the corresponding generating functions. As we will see, the fibers over $0$ and $\infty$ in the base manifold $\mathbb{R}^{2n}$ will play no role in the calculations and so, since $\mathbb{Z}_k$-action is free away from those fibers, the equivariant homology will be obtained essentially by just taking the quotient. In particular we will see that the critical submanifolds, that in the non-equivariant case are diffeomorphic to $S^{2n-1}$, in our case become diffeomorphic to $L_k^{\phantom{k}2n-1}$. This change in the topology of the critical submanifolds is responsible for the different behavior of the homology groups in the equivariant and non-equivariant case and thus for the fact that the contact squeezing of Theorem \ref{squeezing} cannot be performed in the equivariant case.\\
\\
Any bi-invariant partial order defined on a group $G$ gives rise to a numerical invariant, called \textit{relative growth}, for pairs of elements in $G$. The relative growth in turn can be used to associate to $G$ a canonical partially ordered metric space (see \cite{EP}). As observed in \cite{EP}, the notion of relative growth can be applied in particular to the contactomorphism group of an orderable contact manifold, giving rise to an invariant that can be considered as a contact generalization of the rotation number of a diffeomorphism of $S^1$. This was in fact the main motivation of Eliashberg and Polterovich in introducing the notion of orderability in contact topology. It would be interesting to study the relative growth and the geometry of the associated metric space in the special case of lens spaces. However we do not address this question in the present article.\\
\\
This article is organized as follows. In Section \ref{cont_hom} we recall from \cite{mio} the construction of contact homology for domains in $\mathbb{R}^{2n}\times S^1$. In Section \ref{equivariant} we study the equivariant case. In \ref{inv_gf} we prove existence and uniqueness of $\mathbb{Z}_k$-invariant generating functions, and in \ref{equivhom} we use these results to define the equivariant contact homology groups for $\mathbb{Z}_k$-invariant domains of $\mathbb{R}^{2n}\times S^1$, after recalling in \ref{eq} the classical construction of equivariant homology. In Section \ref{calc} we first present in \ref{balls} Traynor's calculation of the symplectic homology of balls \cite{T}, then we show in \ref{balls_eq} how to modify it in the $\mathbb{Z}_k$-equivariant case and finally in \ref{ultima} how to use the obtained results to prove Theorem \ref{non-squeezing}.

\subsection*{Acknowledgments} This work was mostly done during my PhD at Instituto Superior T\'{e}cnico in Lisbon, under the supervision of Miguel Abreu. I thank him for all his mathematical guidance and support. I am also very grateful to Gustavo Granja for his help on algebraic topology and to David Martinez Torres for many discussions, comments and suggestions. My understanding of the material presented in this article improved very much during my mathematical Tour de France in June 2009. In particular I thank Emmanuel Ferrand and David Th\'{e}ret for discussions on equivariant generating functions, and Fran\c{c}ois Laudenbach for pointing out many crucial problems and mistakes in my original approach. I also thank the referee for many useful remarks. My research was supported in Lisbon by an FCT graduate fellowship, and in Nantes by an ANR GETOGA postdoctoral fellowship.

\section{Contact Homology for Domains of $\mathbb{R}^{2n}\times S^1$}\label{cont_hom}

We start with some preliminaries on generating functions, referring to \cite{mio} and the bibliography therein for more details. Let $B$ be a closed manifold, and $S: E\rightarrow\mathbb{R}$ a function defined on the total space of a fiber bundle $p:E\rightarrow B$. We will assume that $dS: E\longrightarrow T^{\ast}E$ is transverse to $N_E:=\{\:(e,\eta)\in T^{\ast}E \;|\; \eta = 0 \;\text{on} \;\text{ker}\;dp\,(e)\:\}$, so that the set $\Sigma_S$ of fiber critical points is a submanifold of $E$ of dimension equal to the dimension of $B$. To any $e$ in $\Sigma_S$ we can associate an element $v^{\ast}(e)$ of $T^{\phantom{p}\ast}_{p(e)}B$ defined by $v^{\ast}(e)\,(X):=dS\,(\widehat{X})$ for $X \in T_{p(e)}B$, where $\widehat{X}$ is any vector in $T_eE$ with $p_{\ast}(\widehat{X})=X$. Then $i_S:\Sigma_S\longrightarrow T^{\ast}B$, $e\mapsto \big(p(e),v^{\ast}(e)\big)$ and $j_S:\Sigma_S\longrightarrow J^1B$, $e\mapsto \big(p(e),v^{\ast}(e),S(e)\big)$ are respectively a Lagrangian and a Legendrian immersion. $S:E \longrightarrow \mathbb{R}$ is called a \textit{generating function} for $i_S\,(\Sigma_S)$ and for $j_S\,(\Sigma_S)$. Note that, since $i_S^{\phantom{S}\ast}\, \lambda_{\text{can}}=d\,(S_{|\Sigma_S})$, $j_S\,(\Sigma_S)$ a lift to $J^1(B)$ of the exact Lagrangian submanifold $i_S\,(\Sigma_S)$ of $T^{\ast}B$. Note also that critical points of $S$ correspond to intersection points of $i_S\,(\Sigma_S)$ with the 0-section and of  $j_S\,(\Sigma_S)$ with the 0-wall. \\
\\
A generating function $S:E \longrightarrow \mathbb{R}$ is called \textit{quadratic at infinity} if $p:E \longrightarrow B $ is a vector bundle and if there exists a non-degenerate quadratic form $\mathcal{Q}_{\infty}: E \longrightarrow \mathbb{R}$ such that $dS-\partial_v\mathcal{Q}_{\infty}: E \longrightarrow E^{\ast}$ is bounded, where $\partial_v$ denotes the fiber derivative. Existence and uniqueness results for generating functions quadratic at infinity have been proved in the symplectic case by Sikorav \cite{S,S2}, using ideas of \cite{LS} and \cite{Ch1}, and by Viterbo \cite{V} and Th\'{e}ret \cite{Th2}. These results were then generalized to the contact case by Chaperon, Chekanov and Th\'{e}ret.

\begin{tthm}[\cite{Ch}, \cite{C}, \cite{Th}]\label{eugfcont}
Any Legendrian submanifold of $J^1B$ contact isotopic to the 0-section has a generating function quadratic at infinity, which is unique up to fiber-preserving diffeomorphism and stabilization.
\end{tthm}

Recall that a stabilization of a generating function $S:E\longrightarrow\mathbb{R}$ defined on the total space $E$ of a vector bundle over $B$ is a function of the form $S'=S+Q: E'=E\oplus F\longrightarrow \mathbb{R}$, where $F\longrightarrow B$ is a vector bundle and $Q:F\longrightarrow\mathbb{R}$ is a non-degenerate quadratic form. A generating function quadratic at infinity $S:E\longrightarrow\mathbb{R}$ is said to be \textit{special} if $E=B\times \mathbb{R}^N$ and $S=S_0+\mathcal{Q}_{\infty}$, where $S_0$ is compactly supported and $\mathcal{Q}_{\infty}$ is the same quadratic form on each fiber. Th\'{e}ret \cite{Th,Th2} proved that any generating function quadratic at infinity can be modified to a special one by applying fiber preserving diffeomorphism and stabilization. In the following we will always consider generating functions  which are quadratic at infinity, and we will assume that they are special whenever this is needed.\\
\\
The theory of generating functions has been applied by Viterbo \cite{V} to the case of compactly supported Hamiltonian symplectomorphisms of $\big(\,\mathbb{R}^{2n}\,,\,\omega=dx\wedge dy\,\big)$. He did this by associating to any $\phi$ in $\text{Ham}^c\,(\mathbb{R}^{2n})$ a Lagrangian submanifold $\Gamma_{\phi}$ of $T^{\ast}S^{2n}$, which is defined to be the compactified image of the composition $\gamma_{\phi}=\sigma\circ\text{gr}_{\phi}: \mathbb{R}^{2n}\rightarrow T^{\ast}\mathbb{R}^{2n}$ where $\text{gr}_{\phi}:\mathbb{R}^{2n}\longrightarrow\overline{\mathbb{R}^{2n}}\times\mathbb{R}^{2n}$ is the graph of $\phi$ and $\sigma: \overline{\mathbb{R}^{2n}}\times\mathbb{R}^{2n}\rightarrow T^{\ast}\mathbb{R}^{2n}$ the symplectomorphism that maps a point $(x,y, X, Y)$ to $(\frac{x+X}{2},\frac{y+Y}{2}, Y-y, x-X)$. This construction can be generalized to the contact case by associating to a contactomorphism $\phi$ of $\big(\mathbb{R}^{2n+1},\xi_0=\text{ker}\,(dz-\frac{ydx-xdy}{2})\big)$ a Legendrian submanifold $\Gamma_{\phi}$ of $J^1\mathbb{R}^{2n+1}$ which is defined as follows. Consider first the graph $\text{gr}_{\phi}:\mathbb{R}^{2n+1}\longrightarrow \mathbb{R}^{2(2n+1)+1}$, $q\mapsto (q,\phi(q),g(q))$ where $g:\mathbb{R}^{2n+1}\longrightarrow \mathbb{R}$ is the function given by $\phi^{\ast}(dz-\frac{ydx-xdy}{2})=e^g(dz-\frac{ydx-xdy}{2})$. Note that $\text{gr}_{\phi}$ is a Legendrian embedding, with respect to the contact structure $\text{ker}\,\big(e^{\theta}(dz-\frac{ydx-xdy}{2})-(dZ-\frac{YdX-XdY}{2})\big)$ on $\mathbb{R}^{2(2n+1)+1}$. We then define $\Gamma_{\phi}$ to be the image of the Legendrian embedding $\gamma_{\phi}=\sigma \circ \text{gr}_{\phi}: \mathbb{R}^{2n+1}\longrightarrow J^1\mathbb{R}^{2n+1}$, where $\sigma:\mathbb{R}^{2(2n+1)+1}\longrightarrow J^1\mathbb{R}^{2n+1}$ is the map 
\begin{equation}\label{formula}
(x,y,z,X,Y,Z,\theta)\mapsto \Big(\frac{e^{\frac{\theta}{2}}x+X}{2},\frac{e^{\frac{\theta}{2}}y+Y}{2},z, Y-e^{\frac{\theta}{2}}y, e^{\frac{\theta}{2}}x-X, e^{\theta}-1, Z-z+\frac{e^{\frac{\theta}{2}}(xY-yX)}{2}\Big).
\end{equation}
Note that $\sigma$ is a contact embedding, sending the diagonal $\{\,x=X\,,\,y=Y\,,\,z=Z\,,\,\theta=0\,\}$ to the 0-section \footnote{There are also other contact embeddings $\mathbb{R}^{2(2n+1)+1}\longrightarrow J^1\mathbb{R}^{2n+1}$ with this property, for example the one introduced by Bhupal \cite{B} and defined by
$(x,y,z,X,Y,Z,\theta)\mapsto \big(x,Y,z, Y-e^{\theta}y, x-X, e^{\theta}-1, xY-XY+Z-z\big)$.
Notice that this map generalizes the symplectomorphism $\overline{\mathbb{R}^{2n}}\times\mathbb{R}^{2n}\rightarrow T^{\ast}\mathbb{R}^{2n}$ defined by $(x,y,X,Y)\mapsto (x,Y,Y-y,x-X)$. Similarly there is also a contact embedding 
$\mathbb{R}^{2(2n+1)+1}\longrightarrow J^1\mathbb{R}^{2n+1}$ generalizing the symplectomorphism $\overline{\mathbb{R}^{2n}}\times\mathbb{R}^{2n}\rightarrow T^{\ast}\mathbb{R}^{2n}$ defined by $(x,y,X,Y)\mapsto (y,X, x-X,Y-y)$. For the purposes of this section we could equivalently use any of these formulas. However formula (\ref{formula}) is the only one which is symmetric with respect to the $\mathbb{Z}_k$-action that we will consider in the next section, and thus which is suitable for developing the $\mathbb{Z}_k$-equivariant theory of contact homology for domains in $\mathbb{R}^{2n}\times S^1$.}. Note moreover that $\Gamma_{\phi}$ can also be written as
$\Gamma_{\phi}=\Psi_{\phi}(\text{0-section})$, with $\Psi_{\phi}$ denoting the local contactomorphism of $J^1\mathbb{R}^{2n+1}$ defined by the diagram
\begin{equation}\label{diagram_c}
\xymatrix{
 \quad \mathbb{R}^{2(2n+1)+1} \quad \ar[r]^{\overline{\phi}} \ar[d]_{\sigma} &
 \quad \mathbb{R}^{2(2n+1)+1} \quad \ar[d]^{\sigma} \\
 \quad J^1\mathbb{R}^{2n+1} \quad \ar[r]_{\Psi_{\phi}} &  \quad J^1\mathbb{R}^{2n+1}}\quad
\end{equation}
where $\overline{\phi}$ is defined by $(p,P,\theta)\mapsto (p,\phi(P), g(P)+\theta)$. This shows in particular that if $\phi$ is contact isotopic to the identity then $\Gamma_{\phi}$ is contact isotopic to the 0-section. Moreover, if $\phi$ is 1-periodic in the $z$-coordinate and compactly supported in the $(x,y)$-plane, i.e. if $\phi$ is a compactly supported contactomorphism of $\mathbb{R}^{2n}\times S^1$, then $\Gamma_{\phi}$ can be seen as a Legendrian submanifold of $J^1(S^{2n}\times S^1)$. By Theorem \ref{eugfcont} we know thus that it has a (special) generating function $S: (S^{2n}\times S^1)\times\mathbb{R}^N\rightarrow \mathbb{R}$, which is unique up to fiber-preserving diffeomorphism and stabilization. In the following, by generating function of a contactomorphism $\phi$ of $\mathbb{R}^{2n}\times S^1$ we will always mean a generating function $S$ for the associated Legendrian submanifold $\Gamma_{\phi}$ of $J^1(S^{2n}\times S^1)$. The crucial property of $S$ is that its critical points correspond to \textit{translated points} of $\phi=(\phi_1,\phi_2,\phi_3)$, i.e. points $q=(x,y,z)$ such that $\phi_1(q)=x$, $\phi_2(q)=y$ and $g(q)=0$ (see \cite{mio}).\\
\\
Let $a$, $b$ be \textit{integer} numbers that are not critical values of $S$. The \textbf{contact homology} $G_{\ast}^{\;\;(a,b]}\,(\phi)$ of $\phi$ with respect to the parameters $a$ and $b$ is defined by
$$G_k^{\;\;(a,b]}\,(\phi):=H_{k+\iota}\,(E^b, E^a) $$
where $E$ denotes the domain of the generating function $S$, $E^a$ and $E^b$ the sublevel sets at $a$ and $b$, and $\iota$ the index of the quadratic at infinity part of $S$. We will also consider the groups $G_{\ast}^{\;\;(a,b]}\,(\phi)$ for $a=-\infty$ or $b=+\infty$ by defining $E^{+\infty}=E$ and $E^{-\infty}=E^c$ for $c$ sufficiently negative. It follows from the uniqueness part of Theorem \ref{eugfcont} that the $G_k^{\;\;(a,b]}\,(\phi)$ are well-defined, i.e. do not depend on the choice of the generating function. Moreover it was proved in \cite{mio} that these groups are invariant by conjugation.

\begin{pprop}\label{conjcont}
For any contactomorphism $\psi$ of $\mathbb{R}^{2n}\times S^1$ isotopic to the identity we have an induced isomorphism 
$$\psi_{\ast}: G_{\ast}^{\;\;(a,b]}\,(\psi\phi\psi^{-1})\longrightarrow G_{\ast}^{\;\;(a,b]}\,(\phi).$$
\end{pprop}

Consider the partial order $\leq$ on the group $\text{Cont}_0^{\phantom{0}c}\,(\mathbb{R}^{2n}\times S^1)$ of compactly supported contactomorphisms of $\mathbb{R}^{2n}\times S^1$ isotopic to the identity which is given by $\phi_0\leq\phi_1$ if $\phi_1\phi_0^{-1}$ is the time-1 flow of a non-negative contact Hamiltonian \footnote{It was proved by Bhupal \cite{B} that $\leq$ defines indeed a partial order, see also \cite{mio}.}. It can be proved (see \cite{mio}) that if $\phi_0\leq\phi_1$ then there are generating functions $S_0$, $S_1: E \longrightarrow \mathbb{R}$ for $\Gamma_{\phi_0}$, $\Gamma_{\phi_1}$ respectively such that $S_0\leq S_1$. Thus inclusion of sublevel sets gives an induced homomorphism 
$\mu_0^{\phantom{0}1}: G_{\ast}^{\;\;(a,b]}\,(\phi_1)\rightarrow G_{\ast}^{\;\;(a,b]}\,(\phi_0)$ which commutes with the isomorphisms given by Proposition \ref{conjcont}. Given a domain $\mathcal{V}$ of $\mathbb{R}^{2n}\times S^1$ we denote by $\text{Cont}_{a,b}^{\phantom{ab}c}\,(\mathcal{V})$ the set of contactomorphisms $\phi$ in $\text{Cont}_0^{\phantom{0}c}\,(\mathbb{R}^{2n}\times S^1)$ with support contained in $\mathcal{V}$ and whose generating function does not have $a$ and $b$ as critical values. Then $\lbrace G_k^{\;\;(a,b]}\,(\phi_i)\rbrace_{\phi_i\in\text{Cont}_{a,b}^{\phantom{ab}c}\,(\mathcal{V})}$ is an inversely directed family of groups, so we can define the \textbf{contact homology} $G_{\ast}^{\;\;(a,b]}\,(\mathcal{V})$ of $\mathcal{V}$ with respect to the values $a$ and $b$ to be the inverse limit of this family. Contact invariance and monotonicity of these groups follow easily from Proposition \ref{conjcont} and from monotonicity of $G_{\ast}^{\;\;(a,b]}\,(\phi)$ with respect to the partial order $\leq$ (see \cite{mio}).

\section{Equivariant Contact Homology}\label{equivariant}

We will develop in this section an equivariant version of the theory of contact homology for domains of 
$\mathbb{R}^{2n}\times S^1$, with respect to the action of $\mathbb{Z}_k$ on $\mathbb{R}^{2n}\times S^1\equiv\mathbb{C}^n\times S^1$ given by $\tau_k: (w_0, \cdots, w_{n-1},z) \mapsto \big(e^{2\pi im_0/k}w_0, \cdots, e^{2\pi im_{n-1}/k}w_{n-1},z \big)$. We first observe that an easy application of the uniqueness theorem for generating functions shows that if $\phi$ is equivariant then every generating function $S:(\mathbb{R}^{2n}\times S^1)\times\mathbb{R}^N\rightarrow\mathbb{R}$ of $\phi$ is $\mathbb{Z}_k$-invariant, with respect to the action $\tau_k$ on $\mathbb{R}^{2n}$, up to fiber preserving diffeomorphism and stabilization. This follows from the next two lemmas.

\begin{llemma}\label{ell}
Let $\phi$ be a $\mathbb{Z}_k$-equivariant contactomorphism of $\mathbb{R}^{2n+1}$. Then the associated Legendrian embedding $\gamma_{\phi}:\mathbb{R}^{2n+1}\longrightarrow J^1\mathbb{R}^{2n+1}$ is also $\mathbb{Z}_k$-equivariant.
\end{llemma}

Here we consider the action of $\mathbb{Z}_k$ on $J^1\mathbb{R}^{2n+1}$ given by $(q,p,\theta)\mapsto \big(\tau_k(q), (\tau_k^{\phantom{k}\ast})^{-1}(p),\theta\big)$. In the following we will denote by $\tau_k$ also the action of the generator of $\mathbb{Z}_k$ on $J^1\mathbb{R}^{2n+1}$.

\begin{proof}
By formula (\ref{formula}) we have
$$\gamma_{\phi}(x,y,z)= \Big(\frac{e^{\frac{g}{2}}x+\phi_1}{2},\frac{e^{\frac{g}{2}}y+\phi_2}{2},z, \phi_2-e^{\frac{g}{2}}y, e^{\frac{g}{2}}x-\phi_1, e^{g}-1, \phi_3-z+\frac{e^{\frac{g}{2}}(x\phi_2-y\phi_1)}{2}\Big)$$
where $g$ is the function satisfying $\phi^{\ast}(dz-\frac{ydx-xdy}{2})=e^g(dz-\frac{ydx-xdy}{2})$ and where $\phi_1$, $\phi_2$ and $\phi_3$ denote respectively the first $n$ components, the second $n$ component and the last component of $\phi$. The $\mathbb{Z}_k$-action on $\mathbb{R}^{2n+1}$ is given by
$$
\tau_k(x,y,z)=\Big(\,x\,\text{cos}(\frac{2\pi}{k})-y\,\text{sin}(\frac{2\pi}{k}) , y\,\text{cos}(\frac{2\pi}{k})+x\,\text{sin}(\frac{2\pi}{k}),z\,\Big)
$$
and on $J^1\mathbb{R}^{2n+1}$ by 
$$
\tau_k(q_1,q_2,q_3,p_1,p_2,p_3,\theta)=\Big(\,q_1\,\text{cos}(\frac{2\pi}{k})-q_2\,\text{sin}(\frac{2\pi}{k}) , q_2\,\text{cos}(\frac{2\pi}{k})+q_1\,\text{sin}(\frac{2\pi}{k}), q_3,$$ $$p_1\,\text{cos}(\frac{2\pi}{k})-p_2\,\text{sin}(\frac{2\pi}{k}) , p_2\,\text{cos}(\frac{2\pi}{k})+p_1\,\text{sin}(\frac{2\pi}{k}), p_3, \theta\,\Big).
$$
Since $\phi$ is $\mathbb{Z}_k$-equivariant, i.e. $\phi\circ\tau_k=\tau_k\circ\phi$, we have that 
$\phi_1\circ\tau_k=\phi_1\,\text{cos}(\frac{2\pi}{k})-\phi_2\,\text{sin}(\frac{2\pi}{k})$, 
$\phi_2\circ\tau_k=\phi_2\,\text{cos}(\frac{2\pi}{k})+\phi_1\,\text{sin}(\frac{2\pi}{k})$ and $\phi_3\circ\tau_k=\phi_3$. Moreover we also have that $g\circ\tau_k=g$. Using this information, a straightforward calculation shows that $\gamma_{\phi}\circ\tau_k=\tau_k\circ\gamma_{\phi}$, i.e. $\gamma_{\phi}$ is $\mathbb{Z}_k$-equivariant.
\end{proof}

\begin{llemma}\label{el}
Let $\phi$ be a (not necessarily $\mathbb{Z}_k$-equivariant) contactomorphism of $\mathbb{R}^{2n+1}$. If $S:\mathbb{R}^{2n+1}\times\mathbb{R}^N\rightarrow\mathbb{R}$ is a generating function for the Legendrian submanifold $\Gamma_{\phi}$ of $J^1\mathbb{R}^{2n+1}$ then the function $\overline{S}:\mathbb{R}^{2n+1}\times\mathbb{R}^N\rightarrow\mathbb{R}$ defined by $\overline{S}(q;\xi)=S\big(\tau_k(q);\xi\big)$ is a generating function for $\tau_k^{\phantom{k}-1}(\Gamma_{\phi})$.
\end{llemma}

\begin{proof}
The set of fiber critical points of $\overline{S}$ is given by
$$\Sigma_{\overline{S}}=\{\,(q;\xi)\in\mathbb{R}^{2n+1}\times\mathbb{R}^N\;|\;\frac{\partial\overline{S}}{\partial\xi}(q;\xi)=0\,\}=\{\,(q;\xi)\in\mathbb{R}^{2n+1}\times\mathbb{R}^N\;|\;\frac{\partial S}{\partial\xi}\big(\tau_k(q);\xi\big)=0\,\}$$
and the Legendrian embedding $i_{\overline{S}}: \Sigma_{\overline{S}} \rightarrow J^1\mathbb{R}^{2n+1}$ maps a point $(q;\xi)$ to 
$$i_{\overline{S}}\,(q;\xi)=\big(q,\frac{\partial\overline{S}}{\partial q}(q;\xi),\overline{S}(q;\xi)\big)=
\Big(q,\tau_k^{\phantom{k}\ast}\Big(\frac{\partial S}{\partial q}\big(\tau_k(q);\xi\big)\Big), S\big(\tau_k(q);\xi\big)\Big).$$
Thus $\overline{S}$ generates
$$\{\;\Big(q,\tau_k^{\phantom{k}\ast}\Big(\frac{\partial S}{\partial q}\big(\tau_k(q);\xi\big)\Big), S(\tau_k(q);\xi)\Big)\;|\;
\frac{\partial S}{\partial\xi}\big(\tau_k(q);\xi\big)=0\;\}$$
$$=
\{\;\Big(\tau_k^{-1}(q),\tau_k^{\phantom{k}\ast}\big(\frac{\partial S}{\partial q}(q;\xi)\big), S(q;\xi)\Big)\;|\;
\frac{\partial S}{\partial\xi}(q;\xi)=0\;\}
= \tau_k^{\phantom{k}-1}(\Gamma_{\phi}).
$$
\end{proof}

If the contactomorphism $\phi$ is $\mathbb{Z}_k$-equivariant then by Lemma \ref{ell} we have $\tau_k^{\phantom{k}-1}(\Gamma_{\phi})=\Gamma_{\phi}$, thus it follows from Lemma \ref{el} and the uniqueness part of Theorem \ref{eugfcont} that $S$ is $\mathbb{Z}_k$-invariant up to fiber-preserving diffeomorphism and stabilization.\\
\\
We will show in the next subsection that it is possible to find a generating function $S:(\mathbb{R}^{2n}\times S^1)\times\mathbb{R}^N\rightarrow\mathbb{R}$ which is truly invariant, but with respect to an action of $\mathbb{Z}_k$ that also rotates the fiber $\mathbb{R}^N$. Moreover we will also prove a $\mathbb{Z}_k$-equivariant uniqueness theorem for $\mathbb{Z}_k$-invariant generating functions of $\phi$.

\subsection{$\mathbb{Z}_k$-Invariant generating functions}\label{inv_gf}

In this subsection we will prove existence and uniqueness of $\mathbb{Z}_k$-invariant generating functions for $\mathbb{Z}_k$-equivariant contactomorphisms of $\mathbb{R}^{2n}\times S^1$. We start by introducing the following terminology. A generating function $S:E\rightarrow\mathbb{R}$ for a $\mathbb{Z}_k$-equivariant contactomorphism $\phi$ of $\mathbb{R}^{2n}\times S^1$ will be called a \textbf{$\mathbb{Z}_k$-invariant generating function} if $E=(\mathbb{R}^{2n}\times S^1)\times \mathbb{R}^{M(2n+1)}$ for some $M$, and $S$ is $\mathbb{Z}_k$-invariant with respect to the total diagonal $\mathbb{Z}_k$-action on $(\mathbb{R}^{2n}\times S^1)\times \mathbb{R}^{M(2n+1)}$. We will say that a $\mathbb{Z}_k$-invariant generating function $S:(\mathbb{R}^{2n}\times S^1)\times \mathbb{R}^{M(2n+1)}\rightarrow \mathbb{R}$ is \textbf{$\mathbb{Z}_k$-quadratic at infinity} if there exists a non-degenerate $\mathbb{Z}_k$-invariant quadratic form $Q_{\infty}$ defined on the total space of the vector bundle $(\mathbb{R}^{2n}\times S^1)\times \mathbb{R}^{M(2n+1)}\rightarrow\mathbb{R}^{2n}\times S^1$ such that $dS-\partial_vQ_{\infty}$ is bounded, where $\partial_v$ denotes the fiber derivative. In other words, $Q_{\infty}$ is a function on $(\mathbb{R}^{2n}\times S^1)\times \mathbb{R}^{M(2n+1)}$ whose restriction to each fiber $\{q\}\times\mathbb{R}^{M(2n+1)}$ is a quadratic form (possibly varying from fiber to fiber) which is invariant with respect to the $\mathbb{Z}_k$-action on $\mathbb{R}^{M(2n+1)}$.\\
\\
We will denote by $\text{Cont}_0^{\phantom{0}\mathbb{Z}_k}(\mathbb{R}^{2n}\times S^1)$ the group of $\mathbb{Z}_k$-equivariant compactly supported contactomorphisms of $\mathbb{R}^{2n}\times S^1$ isotopic to the identity through contactomorphisms of this form.

\begin{prop}\label{exist}
Let $\phi$ be a contactomorphism in $\text{Cont}_0^{\phantom{0}\mathbb{Z}_k}(\mathbb{R}^{2n}\times S^1)$. Then $\phi$ has a $\mathbb{Z}_k$-invariant generating function $\mathbb{Z}_k$-quadratic at infinity.
\end{prop}

\begin{proof}
We will show that the construction of generating functions given in \cite{Ch} (see also \cite{Th} and \cite{mio}) can also be performed in this equivariant setting. We first need to recall the concept of \textit{Greek generating function}. Let $\varphi$ be a contactomorphism of $J^1\mathbb{R}^m$ which is $\mathcal{C}^1$-close to the identity. The Greek generating function of $\varphi$ is a function  $\Phi:\mathbb{R}^m\times(\mathbb{R}^m)^{\ast}\times\mathbb{R}\rightarrow\mathbb{R}$ defined as follows. For $(p,z)\in(\mathbb{R}^m)^{\ast}\times\mathbb{R}$ consider the function $f_{p,z}:\mathbb{R}^m\rightarrow\mathbb{R}$ given by $f_{p,z}(q)=z+pq$. Note that $j^1f_{p,z}:\mathbb{R}^m\rightarrow J^1\mathbb{R}^m$, for $(p,z)$ varying in 
$(\mathbb{R}^m)^{\ast}\times\mathbb{R}$, form a foliation of $J^1\mathbb{R}^m$. Since $\varphi$ is $\mathcal{C}^1$-close to the identity $\varphi\,(j^1f_{p,z})$ is still a section of $J^1\mathbb{R}^m$, and thus it is the 1-jet of a function $\Phi_{p,z}:\mathbb{R}^m\rightarrow\mathbb{R}$. The Greek generating function $\Phi$ is then defined by $\Phi(Q,p,z)=\Phi_{p,z}(Q)$. Consider now a Legendrian submanifold $L$ of $J^1\mathbb{R}^m$ with generating function $S:\mathbb{R}^m\times\mathbb{R}^N\rightarrow\mathbb{R}$, and a compactly supported contact isotopy $\varphi_t$ of $J^1\mathbb{R}^m$ which is $\mathcal{C}^1$-close to the identity and has Greek generating function
$\Phi_t:\mathbb{R}^m\times(\mathbb{R}^m)^{\ast}\times\mathbb{R}\rightarrow\mathbb{R}$. Then the function
$S_t:\mathbb{R}^m\times\big((\mathbb{R}^m)^{\ast}\times\mathbb{R}^m\times\mathbb{R}^N\big)\rightarrow\mathbb{R}$ defined by 
\begin{equation}\label{a}
S_t\,(Q;p,q,\xi):=\Phi_t\,\big(Q,p,S(q;\xi)-pq \big)
\end{equation}
is a generating function for $\varphi_t(L)$. In particular, notice that this formula shows that the function $S_0:\mathbb{R}^m\times\big((\mathbb{R}^m)^{\ast}\times\mathbb{R}^m\times\mathbb{R}^N\big)\rightarrow\mathbb{R}$ defined by 
$S_0\,(Q;p,q,\xi):=S(q;\xi)+p(Q-q)$ is also a generating function for $L$. The construction of a generating function for a contactomorphism $\phi$ of $\mathbb{R}^{2n}\times S^1$ goes now as follows. Let $\phi$ be the time-1 map of a contact isotopy $\phi_t$ and consider a sequence $0=t_0<t_1<\cdots<t_{I-1}<t_I=1$ with all $\phi_{t_i}\phi_{t_{i-1}}^{\phantom{t_{i}}-1}$ close enough to the identity. Then a generating function for $\phi$ is obtained inductively by applying at every step the composition formula (\ref{a}) to a generating function for $\Gamma_{\phi_{t_{i-1}}}\subset J^1(\mathbb{R}^{2n}\times S^1)$ and a Greek generating function for the contact isotopy $\Psi_{\phi_t\phi_{t_{i-1}}^{\phantom{t_{i}}-1}}$ of $J^1(\mathbb{R}^{2n}\times S^1)$, in order to obtain a generating function $S_t$ for $\Psi_{\phi_t\phi_{t_{i-1}}^{\phantom{t_{i}}-1}}\big(\Gamma_{\phi_{t_{i-1}}}\big)=\Gamma_{\phi_t}$. Suppose now that $\phi$ is the time-1 map of a $\mathbb{Z}_k$-equivariant contact isotopy of $\mathbb{R}^{2n}\times S^1$. Then $\phi_t\phi_{t_{i-1}}^{\phantom{t_{i}}-1}$ is also $\mathbb{Z}_k$-equivariant and so, as a straightforward calculation shows, its Greek generating function is $\mathbb{Z}_k$-invariant with respect to the action $\tau_k\times(\tau_k^{\phantom{k}\ast})^{-1}$ on 
$(\mathbb{R}^{2n}\times S^1)\times(\mathbb{R}^{2n}\times S^1)^{\ast}\times\mathbb{R}$. On the other hand we also know by Lemma
\ref{ell} and the induction hypothesis that the generating function of $\Gamma_{\phi_{t_{i-1}}}$ is $\mathbb{Z}_k$-invariant. By (\ref{a}) we see thus that $S_t$ is invariant with respect to the diagonal action of $\mathbb{Z}_k$ on its domain. Note that the domain is of the form $E=(\mathbb{R}^{2n}\times S^1)\times \mathbb{R}^N$ with $N$ a multiple of $2n+1$. We have thus shown that $\phi$ has a $\mathbb{Z}_k$-invariant generating function. It was proved by Th\'{e}ret \cite{Th} that the family $S_t$ can be made quadratic at infinity by an isotopy of fiber preserving diffeomorphisms. More precisely, we can first apply the change of variables $(Q;p,q,\xi)\mapsto(Q;p,q+Q,\xi)$ and define the function $S_t'(Q;p,q,\xi):=S_t(Q;p,q+Q,\xi)$ and then find a fiber preserving diffeomorphism transforming $S_t'$ into a function that outside a compact set coincides with the quadratic form $K(Q;p,q,\xi)=\mathcal{Q}(\xi)-pq$, where $\mathcal{Q}$ is the quadratic form associated to $S$. Since $K$ is a $\mathbb{Z}_k$-invariant quadratic form, this shows that $\phi$ has a $\mathbb{Z}_k$-invariant generating function $\mathbb{Z}_k$-quadratic at infinity. For later purposes it is important to notice that both the change of variables $(Q;p,q,\xi)\mapsto(Q;p,q+Q,\xi)$ and the fiber preserving diffeomorphism constructed by Th\'{e}ret are $\mathbb{Z}_k$-equivariant.
\end{proof}

\noindent
Generating functions that are $\mathbb{Z}_k$-invariant and $\mathbb{Z}_k$-quadratic at infinity will be called simply \textbf{$\mathbb{Z}_k$-generating functions}. Two $\mathbb{Z}_k$-generating functions are said to be \textbf{$\mathbb{Z}_k$-equivalent} if, up to stabilization with a $\mathbb{Z}_k$-invariant quadratic form, we have $S=S'\circ\Psi$ where $\Psi$ is a fiber-preserving diffeomorphism that is equivariant with respect to the total diagonal $\mathbb{Z}_k$-action on base and fibers.

\begin{llemma}\label{special}
Any $\mathbb{Z}_k$-generating function is $\mathbb{Z}_k$-equivalent to a special one.
\end{llemma}

\begin{proof}
The result can be proved as in the non-equivariant setting \cite{Th,Th2}. Let $S:E=(\mathbb{R}^{2n}\times S^1)\times \mathbb{R}^{M(2n+1)}\rightarrow \mathbb{R}$ be a $\mathbb{Z}_k$-generating function with associated quadratic form $\mathcal{Q}:E\rightarrow\mathbb{R}$. For each $q$ in $\mathbb{R}^{2n}\times S^1$ consider the spaces $E_q^{\phantom{q}+}$ and $E_q^{\phantom{q}-}$ associated to $\mathcal{Q}_q:\mathbb{R}^{M(2n+1)}\rightarrow \mathbb{R}$. After stabilization by the opposite of $\mathcal{Q}$ we may assume that $E^+$ and $E^-$ are trivial vector bundles over $\mathbb{R}^{2n}\times S^1$. Using this we can prove that $\mathcal{Q}$ is isomorphic to a quadratic form independent of the base variable. Consider indeed orthonormal sections $e_1,\cdots,e_i:\mathbb{R}^{2n}\times S^1\rightarrow E^-$, where $i$ is the index of $\mathcal{Q}$, and $e_{i+1},\cdots,e_{M(2n+1)}:\mathbb{R}^{2n}\times S^1\rightarrow E^+$, and define a ($\mathbb{Z}_k$-equivariant) fiber preserving diffeomorphism $A$ on $(\mathbb{R}^{2n}\times S^1)\times \mathbb{R}^{M(2n+1)}$ by $A(q;\alpha_1,\cdots,\alpha_{M(2n+1)})=(q;\alpha_1e_1,\cdots,\alpha_{M(2n+1)}e_{M(2n+1)})$. Then $\mathcal{Q}\circ A$ does not depend on the base point $q$. Suppose now that $S:(\mathbb{R}^{2n}\times S^1)\times \mathbb{R}^{M(2n+1)}\rightarrow \mathbb{R}$ is a $\mathbb{Z}_k$-generating function with associated quadratic form $\mathcal{Q}$ independent of the base variable. Then we can use the Moser method as in \cite{Th} to deform $S$ through $\mathbb{Z}_k$-equivariant fiber preserving diffeomorphisms in order to obtain $S\equiv \mathcal{Q}$ outside a compact set. Moreover we can also ask the fiber preserving diffeomorphism to be trivial on any prescribed compact set.
\end{proof}

We now want to prove that all $\mathbb{Z}_k$-generating functions of $\phi$ are $\mathbb{Z}_k$-equivalent. We will do this by following Th\'{e}ret's proof of the uniqueness theorem for generating functions \cite{Th,Th2} and showing how it can be modified to fit in our equivariant setting. We start with the following result.

\begin{llemma}\label{vert_path}
If a path of $\mathbb{Z}_k$-invariant generating functions generates a constant contactomorphism $\phi$ then all the functions in the path are $\mathbb{Z}_k$-equivalent.
\end{llemma}

\begin{proof}
The non-equivariant version of this statement is proved by Th\'{e}ret using the Moser method to obtain a fiber preserving diffeomorphism connecting the begin and end point of the path of generating functions. It is immediate to see that if we apply this method to a path of $\mathbb{Z}_k$-invariant generating functions then the resulting fiber preserving diffeomorphism is $\mathbb{Z}_k$-equivariant.
\end{proof}

As a first application of this result, consider the function $S_0:(\mathbb{R}^{2n}\times S^1)\times(\mathbb{R}^{2n}\times S^1)^{\ast}\times(\mathbb{R}^{2n}\times S^1)\times\mathbb{R}^N\rightarrow\mathbb{R}$ that appears in the proof of Proposition \ref{exist}. Recall that $S_0$ is defined by $S_0(Q;p,q,\xi):=S(q;\xi)+p(Q-q)$, and generates the same contactomorphism as $S:(\mathbb{R}^{2n}\times S^1)\times\mathbb{R}^N\rightarrow\mathbb{R}$. To see that $S_0$ is $\mathbb{Z}_k$-equivalent to $S$ we first apply the change of variables $(Q;p,q,\xi)\mapsto(Q;p,q+Q,\xi)$, that shows that $S_0$ is $\mathbb{Z}_k$-equivalent to the function $S_0'$ defined by $S_0'(Q;p,q,\xi)=S(Q+q;\xi)-pq$. Notice that $S_0'$ can be written more conveniently as $S_0'(q;p,Q,\xi)=S(Q+q;\xi)-pQ$, and can be joined to the function $S_0''(q;p,Q,\xi)=S(q;\xi)-pQ$ by the path of functions $s\mapsto S(sQ+q;\xi)-pQ$, $s\in[0,1)$. Since all functions of this path generate the same contactomorphism, by applying Lemma \ref{vert_path} we see thus that $S_0'$ and $S_0''$ are $\mathbb{Z}_k$-equivalent. Since $S_0'$ is $\mathbb{Z}_k$-equivalent to $S_0$ and $S_0''$ is a $\mathbb{Z}_k$-stabilization of $S$ we see thus that 
$S_0$ is $\mathbb{Z}_k$-equivalent to $S$. This fact, applied to the proof of Proposition \ref{exist} (and to a 1-parameter version of it) shows the following path and homotopy lifting result.

\begin{prop}\label{lift_prop} 
Let $\phi_t$ be an isotopy in $\text{Cont}_0^{\phantom{0}\mathbb{Z}_k}(\mathbb{R}^{2n}\times S^1)$ and suppose that $\phi_0$ has a $\mathbb{Z}_k$-generating function $S_0$. Then there exists a path $S_t'$ of $\mathbb{Z}_k$-generating functions such that $S_0'$ is $\mathbb{Z}_k$-equivalent to $S_0$, and each $S_t'$ generates the corresponding $\phi_t$. Moreover, suppose that $\phi_t^{\phantom{t}s}$ is a homotopy of paths in $\text{Cont}_0^{\phantom{0}\mathbb{Z}_k}(\mathbb{R}^{2n}\times S^1)$ and $S_t^{\phantom{t}0}$ a path of $\mathbb{Z}_k$-generating functions such that each $S_t^{\phantom{t}0}$ generates the corresponding $\phi_t^{\phantom{t}0}$. Then there exists a family $S_t'^{\phantom{t}s}$ such that $S_t'^{\phantom{t}0}$ is $\mathbb{Z}_k$-equivalent to 
$S_t^{\phantom{t}0}$ and each $S_t'^{\phantom{t}s}$ generates the corresponding $\phi_t^{\phantom{t}s}$.
\end{prop}

We can now prove $\mathbb{Z}_k$-uniqueness of generating functions.

\begin{prop}\label{equiv_uniq}
All $\mathbb{Z}_k$-generating functions of a contactomorphism $\phi$ in $\text{Cont}_0^{\phantom{0}\mathbb{Z}_k}(\mathbb{R}^{2n}\times S^1)$ are $\mathbb{Z}_k$-equivalent.
\end{prop}

\begin{proof}
We first prove that the set of elements of $\text{Cont}_0^{\phantom{0}\mathbb{Z}_k}(\mathbb{R}^{2n}\times S^1)$ for which the result holds is stable under isotopy. Let $\phi_t$ be an isotopy in $\text{Cont}_0^{\phantom{0}\mathbb{Z}_k}(\mathbb{R}^{2n}\times S^1)$ and assume that the result holds for $\phi_1$. We want to prove that the same is true for $\phi_0$. Consider two $\mathbb{Z}_k$-generating function $S$ and $T$ of $\phi_0$. By Proposition \ref{lift_prop} we know that the path $\phi_t$ in $\text{Cont}_0^{\phantom{0}\mathbb{Z}_k}(\mathbb{R}^{2n}\times S^1)$ can be lifted to paths of $\mathbb{Z}_k$-generating function $S_t'$ and $T_t'$, each $S_t'$ and $T_t'$ generating the corresponding $\phi_t$, such that $S_0'$ is $\mathbb{Z}_k$-equivalent to $S$ and $T_0'$ to $T$. By assumption, the generating functions $S_1'$ and $T_1'$ of $\phi_1$ are $\mathbb{Z}_k$-equivalent. If we apply the operations relating $S_1'$ and $T_1'$ to the whole path $S_t'$ we obtain a path $S_t''$ lifting $\phi_t$ and such that $S_1''=T_1'$. Thus we have a path of generating functions lifting the path in $\text{Cont}_0^{\phantom{0}\mathbb{Z}_k}(\mathbb{R}^{2n}\times S^1)$ given by following $\phi_t$ from $\phi_0$ to $\phi_1$ and then backwards from $\phi_1$ to $\phi_0$. Since this path in $\text{Cont}_0^{\phantom{0}\mathbb{Z}_k}(\mathbb{R}^{2n}\times S^1)$ is contractible we can lift a homotopy connecting it to the constant path $\phi_0$, obtaining as the time-1 map of the homotopy a path of $\mathbb{Z}_k$-generating functions of $\phi_0$ connecting $S_0''$ to $T_0'$. We can then conclude using Lemma \ref{vert_path}. We have thus reduced the problem into proving that all $\mathbb{Z}_k$-generating functions of the identity are $\mathbb{Z}_k$-equivalent. Let $S:(\mathbb{R}^{2n}\times S^1)\times \mathbb{R}^{M(2n+1)}\rightarrow \mathbb{R}$ be a $\mathbb{Z}_k$-generating function for the identity, and recall that we can assume that it is special, i.e. it coincides outside a compact set with a ($\mathbb{Z}_k$-invariant) quadratic form $\mathcal{Q}_{\infty}$ independent of the base variable. Since $S$ generates the identity, $\Sigma_S$ projects well to $\mathbb{R}^{2n}\times S^1$ and hence it is the graph of a map $v_S:\mathbb{R}^{2n}\times S^1\rightarrow\mathbb{R}^{M(2n+1)}$. Note that $v_S$ is $\mathbb{Z}_k$-equivariant with respect to the $\mathbb{Z}_k$-actions on $\mathbb{R}^{2n}\times S^1$ and $\mathbb{R}^{M(2n+1)}$. Define a fiber preserving diffeomorphism $A$ of $(\mathbb{R}^{2n}\times S^1)\times\mathbb{R}^{M(2n+1)}$ by $A(x,v)=\big(x,v+v_S(x)\big)$ and consider the function $S\circ A$. Note that $S\circ A$ is still $\mathbb{Z}_k$-invariant and that its set of fiber critical points is $\mathbb{R}^{2n}\times S^1\times\{0\}$. Although $S\circ A$ is not special anymore, by Lemma \ref{special} we can make it special without affecting its set of fiber critical points. We have thus shown that every $\mathbb{Z}_k$-generating function of the identity is $\mathbb{Z}_k$-equivalent to some $S:(\mathbb{R}^{2n}\times S^1)\times \mathbb{R}^{M(2n+1)}\rightarrow \mathbb{R}$  which is special and has $\mathbb{R}^{2n}\times S^1\times\{0\}$ as its set of fiber critical points. Note that in particular this means that for any $q\in\mathbb{R}^{2n}\times S^1$ the restriction $S_q$ of $S$ to the fiber above $q$ has only one critical point, at the origin. We can thus apply the generalized Morse lemma to transform $S$ by a $\mathbb{Z}_k$-equivariant fiber preserving diffeomorphism into a function, still denoted by $S$, that coincides with some non-degenerate quadratic form $\mathcal{Q}$ in a neighborhood $\mathcal{U}$ of $\mathbb{R}^{2n}\times S^1\times\{0\}$. Note that the generalized Morse lemma fits well in our $\mathbb{Z}_k$-equivariant setting because its proof is based on the Moser method (see \cite{BH}) and so if we work with $\mathbb{Z}_k$-invariant functions we automatically obtain $\mathbb{Z}_k$-equivariant diffeomorphisms. Note also that for every $q$ in $\mathbb{R}^{2n}\times S^1$ the quadratic form $\mathcal{Q}_q$ is invariant with respect to the $\mathbb{Z}_k$-action on $\mathbb{R}^{M(2n+1)}$. This follows from the fact that the function $S$ coincides with $\mathcal{Q}$ near $0$ and with the $\mathbb{Z}_k$-invariant quadratic form $\mathcal{Q}_{\infty}$ at $\infty$, and that the vertical gradient of $S$ induces a diffeomorphism between $S^{-1}(-\infty)$ and $S^{-1}(-\epsilon)$ and between $S^{-1}(\infty)$ and $S^{-1}(\epsilon)$ for $\epsilon>0$ small. We will now prove that $S$ is globally $\mathbb{Z}_k$-equivalent to the quadratic form $\mathcal{Q}$. Since $S\equiv \mathcal{Q}$ on $\mathcal{U}$, for $\epsilon>0$ small we have an injection $j:\mathcal{Q}^{-1}(-\epsilon)\cap\mathcal{U}\rightarrow S^{-1}(-\epsilon)$. Note that if $j$ extends to a ($\mathbb{Z}_k$-equivariant) diffeomorphism $\mathcal{Q}^{-1}(-\epsilon)\rightarrow S^{-1}(-\epsilon)$ then $S$ and $\mathcal{Q}$ are ($\mathbb{Z}_k$-)equivalent: a ($\mathbb{Z}_k$-equivariant) diffeomorphism is given by following the negative gradient flow of $\mathcal{Q}$ until we reach either $\mathcal{U}$ or $\mathcal{Q}^{-1}(-\epsilon)$ and then back the flow of $S$, after applying $j$ in the second case. The problem is thus reduced into showing that $j:\mathcal{Q}^{-1}(-\epsilon)\cap\mathcal{U}\rightarrow S^{-1}(-\epsilon)$ extends to a $\mathbb{Z}_k$-equivariant fiber preserving diffeomorphism $\mathcal{Q}^{-1}(-\epsilon)\rightarrow S^{-1}(-\epsilon)$. Since $S\equiv \mathcal{Q}_{\infty}$ outside a compact set we can identify $S^{-1}(-\infty)$ with $(\mathbb{R}^{2n}\times S^1)\times \mathbb{R}^{M(2n+1)-i}\times S^{i-1}$ where $i$ is the index of $\mathcal{Q}_{\infty}$. Since the vertical gradient of $S$ induces a fiber preserving diffeomorphism between $S^{-1}(-\infty)$ and $S^{-1}(-\epsilon)$, we identify the embedding $j$ with a ($\mathbb{Z}_k$-equivariant) embedding of $\mathcal{Q}^{-1}(-\epsilon)\cap\mathcal{U}$ into $(\mathbb{R}^{2n}\times S^1)\times \mathbb{R}^{M(2n+1)-i}\times S^{i-1}$. For $q\in\mathbb{R}^{2n}\times S^1$ consider the space $\mathcal{P}_q$ of diffeomorphisms $f_q:\mathcal{Q}_q^{\phantom{q}-1}(-\epsilon)\rightarrow\{q\}\times\mathbb{R}^{k-i}\times S^{i-1}$ extending $j$. It was proved by Th\'{e}ret that $\phi:\mathcal{P}\rightarrow\mathbb{R}^{2n}\times S^1$ is a locally trivial fibration, with contractible fiber. Note also that we have an obvious $\mathbb{Z}_k$-action on $\mathcal{P}$ making $\phi:\mathcal{P}\rightarrow\mathbb{R}^{2n}\times S^1$ equivariant. This action restricts to a $\mathbb{Z}_k$-action on the fiber $\mathcal{P}_0$ over the origin. Take a $\mathbb{Z}_k$-equivariant $f_0$ in $\mathcal{P}_0$, and extend it to $\pi^{-1}(\mathcal{V})$ where $\mathcal{V}$ is a neighborhood of $0$ above which $\pi$ is trivial. Since the $\mathbb{Z}_k$-action is free on $\pi^{-1}(\mathbb{R}^{2n}\times S^1\setminus\{0\})$ we can extend in a $\mathbb{Z}_k$-equivariant way this local section to a global one, and obtain thus a $\mathbb{Z}_k$-equivariant fiber preserving diffeomorphism $\mathcal{Q}^{-1}(-\epsilon)\rightarrow S^{-1}(-\epsilon)$ extending $j$. We have thus shown that every $\mathbb{Z}_k$-generating function of the identity is $\mathbb{Z}_k$-equivalent to a $\mathbb{Z}_k$-invariant non-degenerate quadratic form. This finishes the proof because any two $\mathbb{Z}_k$-invariant non-degenerate quadratic forms are $\mathbb{Z}_k$-equivalent. Indeed, after stabilization we can assume that they are defined on a vector bundle of the same dimension and that they have the same index. It is then easy to find a $\mathbb{Z}_k$-equivariant (linear) fiber preserving diffeomorphism relating them.
\end{proof}

\subsection{Equivariant homology} \label{eq}

The results of the previous subsection will be applied to define the $\mathbb{Z}_k$-equivariant contact homology groups of $\mathbb{Z}_k$-equivariant contactomorphisms of $\mathbb{R}^{2n}\times S^1$. The definition will be based on the classical construction of equivariant homology of a space endowed with a group action, that we will now recall. A standard reference is for instance \cite[III.1]{tD}. Let $X$ be a space endowed with the action of a compact Lie group $G$. Note first that taking the homology of the quotient $X/G$ is not a homotopy invariant operation: this homology changes in general if we replace $X$ by a homotopy equivalent space. The right homotopy invariant substitute for the quotient is constructed as follows. Consider a principal $G$-bundle $p: EG \rightarrow BG$ with contractible total space $EG$. Such a bundle is called a universal principal $G$-bundle, because any other principal $G$-bundle $E \rightarrow B$ can be obtained as a pullback of it. An explicit construction of a universal principal $G$-bundle was given by Milnor \cite{M_art}. Consider the diagonal action of $G$ on the product $EG\times X$. Note that this action is free. The quotient $X_G:=EG \times_G X$ is called the \textit{Borel construction} on $X$, and is well-defined up to homotopy equivalence. The equivariant homology $H_{G,\ast}(X)$ is defined to be the ordinary homology of $X_G$. Note that if the action of $G$ on $X$ is free then the map $\sigma: X_G \rightarrow X/G$ induced by the projection $EG\times X \rightarrow X$ is a homotopy equivalence, thus $H_{G,\ast}(X)=H_{\ast}(X/G)$ in this case.\\
\\
Every $G$-equivariant map $f:X\rightarrow Y$ induces a map $f_G:X_G\rightarrow Y_G$ and so a homomorphism $(f_{G})_{\ast}: H_{G,\ast}(X) \rightarrow H_{G,\ast}(Y)$. Moreover, if $f$ and $g$ are two $G$-homotopic $G$-maps then we have $(f_{G})_{\ast}=(g_{G})_{\ast}$. In other words, $H_{G,\ast}(-)$ is a $G$-homotopy-invariant functor. If $A$ is a $G$-subspace of $X$ then $A_G$ is a subspace of $X_G$. We define the relative equivariant homology $H_{G,\ast}(X,A)$ to be the usual relative homology of the pair $(X_G,A_G)$. Then we have an exact sequence
$$
\cdots\rightarrow H_{G,n}(A) \rightarrow H_{G,n}(X) \rightarrow H_{G,n}(X,A)\rightarrow H_{G,n-1}(A) \rightarrow \cdots
$$
Suppose now that we have a unitary representation $V$ of $G$ of complex dimension $n$, with unit ball and sphere $DV$ and $SV$. Note that the complex structure on $V$ induces a canonical Thom class for the vector bundle $(V\times X)_G \rightarrow X_G$. We have thus an \textit{equivariant Thom isomorphism}
\begin{equation}\label{eq_thom}
H_{G,i}(X) \xrightarrow{\cong} H_{G,i+2n}(DV\times X,SV\times X)
\end{equation}
(see \cite[III (1.10)]{tD}).

\subsection{Equivariant contact homology}\label{equivhom}

Consider a contactomorphism $\phi$ in $\text{Cont}_0^{\phantom{0}\mathbb{Z}_k}(\mathbb{R}^{2n}\times S^1)$ with $\mathbb{Z}_k$-generating function $S: E=(\mathbb{R}^{2n}\times S^1)\times \mathbb{R}^{M(2n+1)}\rightarrow\mathbb{R}$. Since $S$ is $\mathbb{Z}_k$-invariant the action of $\mathbb{Z}_k$ on $E$ restricts to the sublevel sets of $S$, so we can define
$$G_{\mathbb{Z}_k,\,\ast}^{\;\;\;(a,b]}\,(\phi):=H_{\mathbb{Z}_k,\,\ast+\iota}\,(E^b, E^a).$$

\begin{llemma}
$G_{\mathbb{Z}_k,\,\ast}^{\;\;\;(a,b]}\,(\phi)$ is well defined, i.e. does not depend on the choice of the $\mathbb{Z}_k$-generating function $S$.
\end{llemma}

\begin{proof}
By Theorem \ref{equiv_uniq} we just have to check that $\mathbb{Z}_k$-equivariant fiber preserving diffeomorphism and 
$\mathbb{Z}_k$-invariant stabilization do not affect the definition of $G_{\mathbb{Z}_k,\,\ast}^{\;\;\;(a,b]}\,(\phi)$. Independence by $\mathbb{Z}_k$-equivariant fiber preserving diffeomorphism follows directly from the discussion in the previous subsection. Suppose now that $S': (\mathbb{R}^{2n}\times S^1)\times \mathbb{R}^{M(2n+1)}\times \mathbb{R}^{N(2n+1)}\rightarrow\mathbb{R}$ is defined by $S'(q;\xi_1,\xi_2)=S(q;\xi_1)+\mathcal{Q}(\xi_2)$ with $\mathcal{Q}:\mathbb{R}^{N(2n+1)}\rightarrow\mathbb{R}$ a $\mathbb{Z}_k$-invariant quadratic form of index $\iota_{\mathcal{Q}}$. Let $(E^b)'$ and $(E^a)'$ denote the sublevel sets of $S'$ at $a$ and $b$ respectively. We have to show that
$H_{\mathbb{Z}_k,\,\ast}\,(E^b, E^a)=H_{\mathbb{Z}_k,\,\ast+\iota_{\mathcal{Q}}}\,\big((E^b)', (E^a)'\big)$. Note that if $\mathcal{Q}$ is positive definite then the result follows from the fact that the pair $\big((E^b)', (E^a)'\big)$ deformation retracts to $(E^b, E^a)$. On the other hand, if $\mathcal{Q}$ is negative definite then by excision the pair $\big((E^b)', (E^a)'\big)$ is equivalent to $\big(D(E^b), S(E^b)\cup D(E^a)\big)$ and thus the result follows from the equivariant Thom isomorphism (\ref{eq_thom}). The general case follows, after a change of coordinates, by considering separately the positive and negative definite parts of $\mathcal{Q}$.
\end{proof}

We will now check that all the functorial properties of contact homology go through in this equivariant setting.

\begin{prop}\label{equivconjcont}
For any contactomorphism $\psi$ in $\text{Cont}_0^{\phantom{0}\mathbb{Z}_k}(\mathbb{R}^{2n}\times S^1)$ we have an induced isomorphism 
$$\psi_{\mathbb{Z}_k,\,\ast}: G_{\mathbb{Z}_k,\,\ast}^{\;\;\;(a,b]}\,(\psi\phi\psi^{-1})\longrightarrow G_{\mathbb{Z}_k,\,\ast}^{\;\;\;(a,b]}\,(\phi).$$
\end{prop}

\begin{proof}
Let $\psi_t$ be an isotopy connecting $\psi=\psi_t|_{t=1}$ to the identity, and consider a 1-parameter family of $\mathbb{Z}_k$-generating functions $S_t:(\mathbb{R}^{2n}\times S^1)\times\mathbb{R}^{M(2n+1)}\rightarrow\mathbb{R}$
for $\psi_t\phi\psi_t^{-1}$ (see Proposition \ref{lift_prop}). Consider the isotopy 
$\theta_t:S_0^{\phantom{0}-1}(\{a,b\})\rightarrow S_t^{\phantom{t}-1}(\{a,b\})$ defined by following the gradient flow of $S_t$ (see \cite{mio} for more details). Since the $S_t$ are $\mathbb{Z}_k$-invariant it follows that $\theta_t$ is $\mathbb{Z}_k$-equivariant (provided we calculate the gradient flow with respect to a $\mathbb{Z}_k$-invariant metric). We can then extend $\theta_t$ to a $\mathbb{Z}_k$-equivariant isotopy of $(\mathbb{R}^{2n}\times S^1)\times\mathbb{R}^{M(2n+1)}$ and get 
an induced isomorphism $\psi_{\mathbb{Z}_k,\,\ast}: G_{\mathbb{Z}_k,\,\ast}^{\;\;\;(a,b]}\,(\psi\phi\psi^{-1})\longrightarrow G_{\mathbb{Z}_k,\,\ast}^{\;\;\;(a,b]}\,(\phi)$.
\end{proof}

Note also that if $\phi_0$ and $\phi_1$ are $\mathbb{Z}_k$-equivariant contactomorphisms of $\mathbb{R}^{2n}\times S^1$ with $\phi_0\leq\phi_1$ then inclusion of sublevel sets of the generating functions induces, as in the non-equivariant case, a homomorphism 
$(\mu_0^{\phantom{0}1})_{\,\mathbb{Z}_k}: G_{\mathbb{Z}_k,\,\ast}^{\;\;\;(a,b]}\,(\phi_1) \rightarrow G_{\mathbb{Z}_k,\,\ast}^{\;\;\;(a,b]}\,(\phi_0)$. Given a $\mathbb{Z}_k$-invariant domain $\mathcal{V}$ in $\mathbb{R}^{2n}\times S^1$ we define its \textbf{$\mathbb{Z}_k$-equivariant contact homology} $G_{\mathbb{Z}_k,\,\ast}^{\;\;\;(a,b]}\,(\mathcal{V})$  to be the inverse limit with respect to the partial order $\leq$ of the directed family of groups $G_{\mathbb{Z}_k,\,\ast}^{\;\;\;(a,b]}\,(\phi)$, for $\mathbb{Z}_k$-equivariant contactomorphisms $\phi$ of $\mathbb{R}^{2n}\times S^1$ supported in $\mathcal{V}$. Equivariant contact invariance (Theorem \ref{cont_inv_eq}) and monotonicity (Theorem \ref{monot_eq}) of these groups follow easily respectively from Proposition \ref{equivconjcont} and the monotonicity property of the groups $G_{\mathbb{Z}_k,\,\ast}^{\;\;\;(a,b]}\,(\phi)$ with respect to the partial order $\leq$ (see the proofs of the analogous non-equivariant results in \cite{mio} for more details).

\begin{tthm}\label{cont_inv_eq}
For any $\mathbb{Z}_k$-invariant domain $\mathcal{V}$ in $\mathbb{R}^{2n}\times S^1$ and any $\mathbb{Z}_k$-equivariant contactomorphism $\psi$ in $\text{Cont}_0^{\phantom{0}\mathbb{Z}_k}(\mathbb{R}^{2n}\times S^1)$ we have an induced isomorphism 
$\psi_{\mathbb{Z}_k,\,\ast}:G_{\mathbb{Z}_k,\,\ast}^{\;\;(a,b]}\,\big(\psi(\mathcal{V})\big)\longrightarrow G_{\mathbb{Z}_k,\,\ast}^{\;\;(a,b]}\,(\mathcal{V})$.
\end{tthm}

\begin{tthm}\label{monot_eq}
Every inclusion of $\mathbb{Z}_k$-invariant domains induces a homomorphism of the $\mathbb{Z}_k$-equivariant homology groups (reversing the order), with the following functorial properties: 
\begin{enumerate}
\renewcommand{\labelenumi}{(\roman{enumi})}
\item
If $\mathcal{V}_1\subset\mathcal{V}_2\subset\mathcal{V}_3$ then the following diagram commutes
\begin{displaymath}
\xymatrix{
 G_{\mathbb{Z}_k,\,\ast}^{\;\;(a,b]}\,(\mathcal{V}_3) \ar[r] \ar[rd] &
 G_{\mathbb{Z}_k,\,\ast}^{\;\;(a,b]}\,(\mathcal{V}_2) \ar[d] \\
 & G_{\mathbb{Z}_k,\,\ast}^{\;\;(a,b]}\,(\mathcal{V}_1).}
\end{displaymath} 
\item
If $\mathcal{V}_1\subset\mathcal{V}_2$, then for any $\mathbb{Z}_k$-equivariant contactomorphism $\psi$ the following diagram commutes
\begin{displaymath}
\xymatrix{
 G_{\mathbb{Z}_k,\,\ast}^{\;\;(a,b]}\,(\mathcal{V}_2) \ar[r] &
 G_{\mathbb{Z}_k,\,\ast}^{\;\;(a,b]}\,(\mathcal{V}_1)  \\
 G_{\mathbb{Z}_k,\,\ast}^{\;\;(a,b]}\,\big(\psi(\mathcal{V}_2)\big) \ar[u]^{\psi_{\mathbb{Z}_k,\,\ast}} \ar[r] &  
 G_{\mathbb{Z}_k,\,\ast}^{\;\;(a,b]}\,\big(\psi(\mathcal{V}_1)\big) \ar[u]_{\psi_{\mathbb{Z}_k,\,\ast}}.}
\end{displaymath} 
\end{enumerate}
\end{tthm}

Note that the results in \ref{inv_gf} can be specialized to the symplectic case in order prove existence and uniqueness of $\mathbb{Z}_k$-generating functions for compactly supported symplectomorphisms of $\mathbb{R}^{2n}$ that are generated by a $\mathbb{Z}_k$-invariant Hamiltonian. Thus we can also define, as in the contact case, equivariant symplectic homology groups for $\mathbb{Z}_k$-invariant domains of $\mathbb{R}^{2n}$. The relation between the equivariant symplectic homology of a domain $\mathcal{U}$ of $\mathbb{R}^{2n}$ and the equivariant contact homology of its prequantization $\mathcal{U}\times S^1$ is given by the following theorem, that can be proved as in the non-equivariant case (see \cite{mio}).

\begin{tthm}\label{lift3_eq}
For any $\mathbb{Z}_k$-invariant domain $\mathcal{U}$ of $\mathbb{R}^{2n}$ we have $$G_{\mathbb{Z}_k,\,\ast}^{\;\;(a,b]}\,(\mathcal{U}\times S^1)=G_{\mathbb{Z}_k,\,\ast}^{\;\;(a,b]}\,(\mathcal{U})\otimes H_{\ast}(S^1).$$ Moreover, this correspondence is functorial in the following sense. Let $\mathcal{U}_1$, $\mathcal{U}_2$ be $\mathbb{Z}_k$-invariant domains in $\mathbb{R}^{2n}$ with $\mathcal{U}_1\subset\mathcal{U}_2$, and for $i=1,2$ identify $G_{\mathbb{Z}_k,\,\ast}^{\;\;(a,b]}\,(\mathcal{U}_i\times S^1)$ with $G_{\mathbb{Z}_k,\,\ast}^{\;\;(a,b]}\,(\mathcal{U}_i)\otimes H_{\ast}(S^1)$. Then the homomorphism $G_{\mathbb{Z}_k,\,\ast}^{\;\;(a,b]}\,(\mathcal{U}_2\times S^1)\rightarrow G_{\mathbb{Z}_k,\,\ast}^{\;\;(a,b]}\,(\mathcal{U}_1\times S^1)$ induced by the inclusion $\mathcal{U}_1\times S^1\hookrightarrow\mathcal{U}_2\times S^1$ is given by $\mu_{\mathbb{Z}_k}\otimes \text{id}$, where $\mu_{\mathbb{Z}_k}: G_{\mathbb{Z}_k,\,\ast}^{\;\;(a,b]}\,(\mathcal{U}_2)\rightarrow G_{\mathbb{Z}_k,\,\ast}^{\;\;(a,b]}\,(\mathcal{U}_1)$ is the homomorphism induced by $\mathcal{U}_1\hookrightarrow\mathcal{U}_2$.
\end{tthm}

\section{Calculations for balls and the equivariant contact non-squeezing theorem} \label{calc}

In this last section we will calculate the equivariant symplectic homology of balls in $\mathbb{R}^{2n}$ and use this calculation, together with Theorem \ref{lift3_eq}, to prove the equivariant non-squeezing theorem stated in the introduction (Theorem \ref{non-squeezing}). We start by recalling Traynor's calculation of the symplectic homology of balls.

\subsection{Symplectic homology of balls}\label{balls}

In \cite{T} Traynor calculated the symplectic homology of ellipsoids in $\mathbb{R}^{2n}$. We will present here her calculation in the special case of balls $B(R)$ and of intervals of the form $(a,\infty]$. We will see in the next subsection how this calculation has to be modified in the $\mathbb{Z}_k$-equivariant case. We first define an unbounded ordered sequence supported in $B(R)$. Let $\rho: [0,\infty)\rightarrow \mathbb{R}$ be a function supported in $[0,1]$ and such that $\rho''\geq 0$, $\rho''(m)>0$ for $m$ with $\rho'(m)\in R\,\mathbb{N}$, and $\rho'|_{[0,\delta]}\equiv c<0$ for some $\delta>0$. Given such a function we define a Hamiltonian symplectomorphism $\phi_{\rho}$ of $\mathbb{R}^{2n}$ to be the time-1 map of the Hamiltonian flow of $H_{\rho}:\mathbb{R}^{2n}\rightarrow\mathbb{R}$, $(w_0,\cdots,w_{n-1})\mapsto\rho\big(\frac{\pi}{R}\Sigma|w_i|^2\big)$, thus $\phi_{\rho}(w)=e^{i\frac{2\pi}{R}\rho'(\frac{\pi}{R}|w|^2)}w$. We then take a sequence of functions $\rho_{\kappa}: [0,\infty)\rightarrow \mathbb{R}$ of the above form in such a way that $\lim_{\kappa \to \infty} \rho_{\kappa}(0) = \infty$, $\lim_{\kappa \to \infty} \rho_{\kappa}'(0) = -\infty$ and the associated $\phi_{\rho_1}\leq\phi_{\rho_2}\leq\phi_{\rho_3}\leq\cdots$ form an unbounded ordered sequence (supported in $B(R)$).\\
\\
Consider now generating functions $S_{\rho_{\kappa}}$ for $\phi_{\rho_{\kappa}}$, $\kappa=1,2,3,\cdots$. Let $m_j\in(0,1)$ be defined by $\rho_{\kappa}'(m_j)=-jR$ for $j=1,\cdots,\nu$. It can be shown (see \cite{T}) that for $j=1,\cdots,\nu$ the fixed points set $Z_j=\{\frac{\pi}{R}|w|^2=m_j\}$ corresponds to a non-degenerate critical submanifold of $S_{\rho_{\kappa}}$ (diffeomorphic to $S^{2n-1}$) with critical value $c_{j,\kappa}=jRm_j+\rho_{\kappa}(m_j)$ and index $2jn+\iota$, and that the fixed point $Z_0=\{0\}$ corresponds to a non-degenerate critical point of $S_{\rho_{\kappa}}$ with critical value $\rho_{\kappa}(0)$ and index $2(\nu+1)n+\iota$. Here $\iota$ is the index of the quadratic at infinity part of $S_{\rho_{\kappa}}$. Moreover, $c_{j,\kappa}<jR$ for all $j$ and $\kappa$, and $\lim_{\kappa \to \infty} c_{j,\kappa} = jR$. Note also that from the point of view of Morse theory the critical submanifold corresponding to $Z_{\infty}=\mathbb{R}^{2n}\setminus B(R)$ behaves as a non-degenerate critical point of index $\iota$ (and critical value $0$).\\
\\
We can now calculate $G_{\ast}^{\;\;(a,\infty]}\,(B(R))=\varprojlim G_{\ast}^{\;\;(a,\infty]}\,(\phi_{\rho_{\kappa}})$. Note first that $G_{\ast}^{\;\;(a,\infty]}\,(B(R))=G_{\ast}^{\;\;(a,\infty]}\,(\phi_{\rho_{\kappa}})$ for $\kappa$ arbitrarily big. We will then use the following facts. 
\begin{enumerate}
\item
For $a_1<a_2$ we have an exact sequence
\begin{displaymath}
\xymatrix{\cdots \ar[r] &  G_{\ast}^{\;\;(a_1,a_2]}\,(\phi_{\rho_{\kappa}})\ar[r] &  G_{\ast}^{\;\;(a_1,\infty]}\,(\phi_{\rho_{\kappa}})\ar[r] & G_{\ast}^{\;\;(a_2,\infty]}\,(\phi_{\rho_{\kappa}})\ar[r] \ar[r] & G_{\ast-1}^{\;\;(a_1,a_2]}\,(\phi_{\rho_{\kappa}}) \ar[r] & \cdots }
\end{displaymath} 
coming from the exact sequence of the triple $E^{a_1}\subset E^{a_2} \subset E$.
\item If the interval $(a_1,a_2]$ only contains the critical value $c_{j,\kappa}$ then
$$
G_{\ast}^{\;\;(a_1,a_2]}\,(\phi_{\rho_{\kappa}})=H_{\ast-2jn}(S^{2n-1})=\left\{ 
\begin{array}{l l}
 \mathbb{Z}_k & \quad \text{if}\;\; \ast=2jn\,,\,2(j+1)n-1\\
 0            & \quad \text{otherwise.}
\end{array} \right. 
$$
Indeed, by Morse-Bott theory we know that passing a non-degenerate critical submanifold of index $\lambda$ changes the topology of the sublevel sets by the attachment of a $\lambda$-disk bundle over the critical submanifold. By the Thom isomorphism, and since all critical submanifolds are diffeomorphic to $S^{2n-1}$, the relative homology of the sublevel sets is thus given by the homology of $S^{2n-1}$ shifted by $\lambda$.   
\item By the Thom isomorphism we have
$$G_{\ast}^{\;\;(-\infty,\infty]}\,(\phi_{\rho_{\kappa}})=H_{\ast+\iota}(E,E^{-\infty})=H_{\ast}(S^{2n}).$$
\end{enumerate}
Using (1)-(3) we get the following result.

\begin{tthm}[\cite{T}]\label{thm_balls}
Consider $B(R)\subset\mathbb{R}^{2n}$ and let $a$ be a positive real number. Then the symplectic homology of $B(R)$ with $\mathbb{Z}_k$-coefficients is given by
$$ G_{\ast}^{\;\;(a,\infty]}\,\big(B(R)\big)  = \left\{ 
\begin{array}{l l}
 \mathbb{Z}_k & \quad \text{if}\;\; \frac{a}{l}\leq R <\frac{a}{l-1}\\
 0            & \quad \text{otherwise}
\end{array} \right. $$
for $\ast=2nl$, where $l$ is any positive integer. For all other values of $\ast$ the corresponding homology groups are zero. Moreover, given $R$, $R'$ with $\frac{a}{l}\leq R < R'<\frac{a}{l-1}$, the homomorphism $G_{\ast}^{\;\;(a,\infty]}\,\big(B(R')\big)\longrightarrow G_{\ast}^{\;\;(a,\infty]}\,\big(B(R)\big)$ induced by the inclusion $B(R)\subset B(R')$ is an isomorphism.
\end{tthm}

\subsection{Equivariant homology of balls}\label{balls_eq}

We continue to work in the setting of the previous subsection. Note that all the $\rho_{\kappa}$ are $\mathbb{Z}_k$-equivariant, so $G_{\mathbb{Z}_k,\,\ast}^{\;\;\;(a,\infty]}\,(B(R))=\varprojlim G_{\mathbb{Z}_k,\,\ast}^{\;\;\;(a,\infty]}\,(\phi_{\rho_{\kappa}})$. As before we have that $G_{\mathbb{Z}_k,\,\ast}^{\;\;\;(a,\infty]}\,(B(R))=G_{\mathbb{Z}_k,\,\ast}^{\;\;\;(a,\infty]}\,(\phi_{\rho_{\kappa}})$ for $\kappa$ big enough. Note that for $a_1<a_2$ we still have an exact sequence
\begin{displaymath}
\xymatrix{\cdots \ar[r] &  G_{\mathbb{Z}_k,\,\ast}^{\;\;\;(a_1,a_2]}\,(\phi_{\rho_{\kappa}})\ar[r] &  G_{\mathbb{Z}_k,\,\ast}^{\;\;\;(a_1,\infty]}\,(\phi_{\rho_{\kappa}})\ar[r] & G_{\mathbb{Z}_k,\,\ast}^{\;\;\;(a_2,\infty]}\,(\phi_{\rho_{\kappa}})\ar[r] \ar[r] & G_{\mathbb{Z}_k,\,\ast-1}^{\;\;\;(a_1,a_2]}\,(\phi_{\rho_{\kappa}}) \ar[r] & \cdots }
\end{displaymath} 
Note also that if the interval $(a_1,a_2]$ contains only the critical value $c_{j,\kappa}$ then, by the equivariant Thom isomorphism (\ref{eq_thom}) and since the action is free near the critical submanifold with critical value $c_{j,\kappa}$,  $G_{\mathbb{Z}_k,\,\ast}^{\;\;\;(a,\infty]}\,(\phi_{\rho_{\kappa}})$ is given by the homology of the quotient of the critical submanifold by the action. Thus we have
$$
G_{\mathbb{Z}_k,\,\ast}^{\;\;\;(a_1,a_2]}\,(\phi_{\rho_{\kappa}})=H_{\ast-2jn}(L^{2n-1})=\left\{ 
\begin{array}{l l}
 \mathbb{Z}_k & \quad \text{if}\;\; \ast=2jn\,,\,2jn+1\,,\,\cdots\,,\,2(j+1)n-1\\
 0            & \quad \text{otherwise.}
\end{array} \right. 
$$
Moreover we have $G_{\mathbb{Z}_k,\,\ast}^{\;\;\;(\rho_{\kappa}(0)-\epsilon,\infty]}\,(B(R))=0$ for all $\ast$ smaller than the index of the critical value $\rho_{\kappa}(0)$. Using these observations we can now prove the following theorem.

\begin{tthm}\label{thm_balls_eq}
The $\mathbb{Z}_k$-equivariant symplectic homology of $B(R)$ with $\mathbb{Z}_k$-coefficients and with respect to the interval $(a,\infty]$  for $a>0$ is given by
$$G_{\mathbb{Z}_k,\,\ast}^{\;\;\;(a,\infty]}\,\big(B(R)\big)  = \left\{ 
\begin{array}{l l}
 \mathbb{Z}_k & \quad \text{if}\;\; R\geq\frac{a}{l} \quad \text{and}\;\;2nl\leq\ast<2n(l+1)-1\\
 0            & \quad \text{otherwise}
\end{array} \right. $$
where $l$ is any positive integer. Moreover, given $R$, $R'$ with $\frac{a}{l}\leq R < R'$, the homomorphism $G_{\mathbb{Z}_k,\,\ast}^{\;\;\;(a,\infty]}\,\big(B(R')\big)\longrightarrow\; G_{\mathbb{Z}_k,\,\ast}^{\;\;\;(a,\infty]}\,\big(B(R)\big)$ induced by the inclusion $B(R)\subset B(R')$ is in fact an isomorphism.
\end{tthm}

For our application in the next subsection it will be sufficient to consider homology groups of degree $\ast=2nl$. In this case we have that $G_{\mathbb{Z}_k,\,\ast}^{\;\;\;(a,\infty]}\,\big(B(R)\big)$ does not vanishes for all $R\geq\frac{a}{l}$, in contrast with the case of $G_{\ast}^{\;\;\;(a,\infty]}\,\big(B(R)\big)$ that for $\ast=2nl$ does not vanishes only in the subinterval $\frac{a}{l}\leq R <\frac{a}{l-1}$. This difference is responsible for the different form of the non-squeezing theorem in the equivariant and non-equivariant \cite{EKP, mio} case.

\begin{proof}
The only case that does not follow immediately from the discussion above is when $\ast=2nl$ and $a\leq(l-1)R$. In this case we have $G_{\mathbb{Z}_k,\,\ast}^{\;\;\;(a,\infty]}\,\big(B(R)\big)=G_{\mathbb{Z}_k,\,\ast}^{\;\;\;(a,a']}\,\big(B(R)\big)$
with $lR<a'<(l+1)R$, but the exact sequences for $a<a''<\infty$ and $a''<a'<\infty$ with $(l-1)R<a''<lR$ do not allow us to conclude, since both $G_{\mathbb{Z}_k,\,\ast}^{\;\;\;(a'',a']}\,\big(B(R)\big)$ and $G_{\mathbb{Z}_k,\,\ast-1}^{\;\;\;(a,a'']}\,\big(B(R)\big)$ do not vanish. To get the result we will follow the approach of Morse homology for generating functions, as introduced by Milinkovi\'{c} \cite{M,Mil97}. In order to turn the generating function of $\rho_{\kappa}$ into a ($\mathbb{Z}_k$-invariant) Morse function we will perturb it by a $\mathbb{Z}_k$-invariant Morse function $f$ on $S^{2n-1}$ with $k$ critical points $\{\,a_{0,j}^0,\cdots,a_{k-1,j}^0\,\}$ of index $2j$ and $k$ critical points $\{\,a_{0,j}^1,\cdots,a_{k-1,j}^1\,\}$ of index $2j+1$ for each $j=0,\cdots,n-1$ (see \cite[p26]{Milnor}). We can assume the critical points are numbered in such a way that the $\mathbb{Z}_k$-action sends $a_{\kappa,j}^{\nu}$ to $a_{\kappa+m_j,j}^{\nu}$. Then, after  identifying $a_{\kappa,j}^{\nu}$ with $T^k$, the Morse complex of $f$ is 
\begin{align*}
0\; \longrightarrow 
\dfrac{\mathbb{Z}_k[T]}{T^k-1} &\xrightarrow{\cdot(T^{m_{n-1}}-1)}
\dfrac{\mathbb{Z}_k[T]}{T^k-1} \xrightarrow{(\ast)}
\dfrac{\mathbb{Z}_k[T]}{T^k-1} \xrightarrow{\cdot(T^{m_{n-2}}-1)}
\dfrac{\mathbb{Z}_k[T]}{T^k-1} \xrightarrow{(\ast)}\\
&\cdots\longrightarrow
\dfrac{\mathbb{Z}_k[T]}{T^k-1} \xrightarrow{\cdot(T^{m_0}-1)}
\dfrac{\mathbb{Z}_k[T]}{T^k-1} \longrightarrow \;0
\end{align*} 
and the $\mathbb{Z}_k$-action for generators of index $2j$ and $2j+1$ is given by multiplication by $T^{m_j}$. Since the above complex calculates $H_{\ast}(S^{2n-1})$ we know that the boundary maps $(\ast)$ must be multiplication by $T^{k-1}+\cdots+T+1$. Now we use the function $f$ to perturb $S_{\rho_{\kappa}}$ inside small tubular neighborhoods of the $S^{2n-1}$ critical submanifolds of index $2nl$ and $2n(l-1)$. Then the complex calculating 
$G_{\ast}^{\;\;(a,a']}\,\big(B(R)\big)$ is
\begin{align*}
0\; \longrightarrow 
\dfrac{\mathbb{Z}_k[T]}{T^k-1} \xrightarrow{\cdot(T^{m_{n-1}}-1)}
&\dfrac{\mathbb{Z}_k[T]}{T^k-1} \xrightarrow{\cdot(T^{k-1}+\cdots+T+1)}
\dfrac{\mathbb{Z}_k[T]}{T^k-1} \xrightarrow{\cdot(T^{m_{n-2}}-1)}
\dfrac{\mathbb{Z}_k[T]}{T^k-1} \xrightarrow{\cdot(T^{k-1}+\cdots+T+1)}\\
&\cdots\; \longrightarrow 
\dfrac{\mathbb{Z}_k[T]}{T^k-1} \xrightarrow{\cdot(T^{m_0}-1)}
\dfrac{\mathbb{Z}_k[T]}{T^k-1} \xrightarrow{(\ast\ast)}\\
\dfrac{\mathbb{Z}_k[T]}{T^k-1} \xrightarrow{\cdot(T^{m_{n-1}}-1)}
&\dfrac{\mathbb{Z}_k[T]}{T^k-1} \xrightarrow{\cdot(T^{k-1}+\cdots+T+1)}
\dfrac{\mathbb{Z}_k[T]}{T^k-1} \xrightarrow{\cdot(T^{m_{n-2}}-1)}
\dfrac{\mathbb{Z}_k[T]}{T^k-1} \xrightarrow{\cdot(T^{k-1}+\cdots+T+1)}\\
&\cdots\; \longrightarrow 
\dfrac{\mathbb{Z}_k[T]}{T^k-1} \xrightarrow{\cdot(T^{m_0}-1)}
\dfrac{\mathbb{Z}_k[T]}{T^k-1} \xrightarrow\;0\\
\end{align*} 
Since, by Theorem \ref{thm_balls}, we have $G_{\ast}^{\;\;(a,a']}\,\big(B(R)\big)=0$ for $\ast=2nl$ the map $(\ast\ast)$ must be multiplication by $T^{k-1}+\cdots+T+1$. Note that the $\mathbb{Z}_k$-action is free on $S_{\rho_{\kappa}}^{\phantom{\rho_{\kappa}}-1}\big((a,a']\big)$, so that the equivariant homology $G_{\mathbb{Z}_k,\,\ast}^{\;\;\;(a,a']}\,\big(B(R)\big)$ is obtained by taking the homology of the quotient of the above chain complex by the $\mathbb{Z}_k$-action. We obtain thus that $G_{\mathbb{Z}_k,\,\ast}^{\;\;\;(a,a']}\,\big(B(R)\big)=\mathbb{Z}_k$ for  $\ast=2nl$, as we wanted.
\end{proof}

\subsection{The equivariant non-squeezing theorem}\label{ultima}

We will now use Theorems \ref{lift3_eq} and \ref{thm_balls_eq} to prove Theorem \ref{non-squeezing}. Suppose that there is a $\mathbb{Z}_k$-equivariant contact isotopy of $\mathbb{R}^{2n}\times S^1$ squeezing $\widehat{B(R)}$ into $\widehat{B(R')}$ for $R'$ arbitrarily small. Then in particular we can find a $\mathbb{Z}_k$-equivariant contactomorphism $\psi$ of $\mathbb{R}^{2n}\times S^1$, supported in some big $\widehat{B(R'')}$ and isotopic to the identity through $\mathbb{Z}_k$-equivariant contactomorphisms, such that $\psi\,\big(\widehat{B(R)}\big)\subset\widehat{B(R')}$. Consider the following diagram
\begin{displaymath}
\xymatrix{
 G_{\mathbb{Z}_k,\,\ast}^{\;\;(1,\infty]}\,(\widehat{B(R'')}) \ar[r] &
 G_{\mathbb{Z}_k,\,\ast}^{\;\;(1,\infty]}\,(\widehat{B(R)})  \\
 G_{\mathbb{Z}_k,\,\ast}^{\;\;(1,\infty]}\,(\widehat{B(R'')}) \ar[u]^{\psi_{\mathbb{Z}_k,\,\ast}} \ar[r] & 
 G_{\mathbb{Z}_k,\,\ast}^{\;\;(1,\infty]}\,(\widehat{B(R')}) \ar[r] & 
 G_{\mathbb{Z}_k,\,\ast}^{\;\;(1,\infty]}\,\big(\psi(\widehat{B(R)})\big) \ar[ul]_{\psi_{\mathbb{Z}_k,\,\ast}}}.
\end{displaymath}
where the horizontal maps are homomorphisms induced by the inclusion of the corresponding domains (see Theorem \ref{monot_eq}) and the vertical ones are isomorphisms induced by $\psi$ (see Theorem \ref{cont_inv_eq}). Since $R'$ is arbitrarily small we can find a positive integer $l$ with $R'<\frac{1}{l}<R$. Consider now $\ast=2nl$. Then by Theorem \ref{lift3_eq} and \ref{thm_balls_eq} we know that $G_{\mathbb{Z}_k,\,\ast}^{\;\;(1,\infty]}\,(\widehat{B(R')})=0$, and that the horizontal map on the top is not the 0-homomorphism. Thus the diagram yields a contradiction and we get the desired result.\\
\\

\end{document}